\documentclass[11pt]{article}
\usepackage{a4,amsmath,amssymb,amsthm,amscd,url}

\def \N {\mathbb{N}}
\def \Z {\mathbb{Z}}

\def \A {\mathcal{A}}

\def \Fwd {\operatorname{Fwd}}
\def \rat {{\cal RWP}}

\newtheorem{theorem}{Theorem}

\newtheorem{lemma}[theorem]{Lemma}
\newtheorem{corollary}[theorem]{Corollary}

\begin{document}

\begin{center}
{\sc\large Inverse semigroups with rational word problem are finite}

{\small Tara Brough}
\footnote{School of Mathematics and Statistics, Mathematical Institute, North Haugh, St Andrews, Fife KY16 9SS, Scotland.  Email: tara@mcs.st-andrews.ac.uk.} 
\end{center}

\abstract{This note proves a generalisation to inverse semigroups of Anisimov's theorem that a group has regular word problem if 
and only if it is finite, answering a question of Stuart Margolis.  The notion of word problem used is the two-tape word problem --
the set of all pairs of words over a generating set for the semigroup which both represent the same element.}

\section{Introduction}

The word problem of a semigroup is, informally, the problem of deciding whether two words over some finite generating set 
represent the same element of the semigroup.  This problem was shown to be undecidable in general for finitely presented 
semigroups by Post \cite{Post},
prior to the establishment of the undecidability of the word problem for finitely presented groups by Novikov \cite{Nov} and Boone \cite{Boo}.

Recent research on word problems of groups and semigroups has centred on word problems which are not only decidable, but 
in some sense `easily' decidable, for example recognisable by finite or pushdown automata.

Anisimov \cite{Ani} introduced the perspective of considering the word problem of a group $G$ with respect to a finite generating set $X$ 
as a formal language, namely the language $W(G,X)$ of all strings over $X$ representing the identity element of $G$.  
Anisimov showed that the word problem of a group is regular (recognised by a finite automaton) if and only if it is finite \cite{Ani}, 
a result known as Anisimov's Theorem.  
He also proposed as an open problem the classification of groups with
context-free word problem, resolved by Muller and Schupp \cite{MulSch} (together with a slightly later result of Dunwoody \cite{Dun}): 
they are precisely the finitely generated virtually free groups.

For a group $G$, the language $W(G,X)$ captures all the information 
required to determine whether two strings over $X$ represent the same element of $G$, but for semigroups a different notion is required, due to the 
lack of inverses and the relative structural unimportance of an identity element, even when present.
The notion of word problem considered in this note is the \emph{two-tape word problem}.  This is the set of all pairs of strings
$(u,v)$ over a generating set for the semigroup $S$ such that $u$ and $v$ represent the same element of $S$.
A semigroup has \emph{rational word problem} if its two-tape word problem is accepted by a two-tape asynchronous finite automaton
(defined in Section~\ref{rat} below).  We denote the class of semigroups with rational word problem by $\rat$. 

In an article establishing much of the basic theory of semigroups with rational word problem, Neunh\"offer, Pfeiffer and Ru\v{s}kuc~\cite{NPR} 
showed that the only groups in $\rat$ are the finite groups,
and hence rational word problem is a genuine generalisation to semigroups of regular group word problem.
The class $\rat$ contains many infinite semigroups, free semigroups being one very straightforward example.
There is as yet no conjecture as to the general structure of semigroups in $\rat$.  All known examples are `not far from free',
but it is difficult to pin down exactly what we mean by this.   

Following a talk by the author on Green's relations for semigroups with rational word problem
(work in progress jointly with Pfeiffer; some results are contained in Pfeiffer's doctoral thesis \cite{Pfe}),
Stuart Margolis asked whether $\rat$ contains any infinite inverse semigroups.
Inverse semigroups are sufficiently far from being free that a negative answer seemed most likely, and the aim of this note 
is to show that indeed the only inverse semigroups in $\rat$ are finite.  This generalises Anisimov's Theorem 
to inverse semigroups.  The proof is elementary and requires few ingredients.  We introduce some basic theory of the class $\rat$ 
in Section~\ref{rat}, inverse semigroups and particularly the classification of the monogenic inverse
semigroups in Section~\ref{mono}, and finally prove the result in Section~\ref{proof}.

\section{Semigroups with rational word problem}\label{rat}

For a semigroup $S$ generated by a finite set $A$, the \emph{two-tape word problem} (henceforth referred to as the word problem) 
of $S$ with respect to $A$ is the relation
\[ \iota(S,A) = \{ (u,v)\in A^+\times A^+ \mid u =_S v \}, \]
where $A^+$ denotes the set of all non-empty strings over $A$.
The notation comes from the fact that this is the lift of the equality relation (often denoted $\iota$) on $S$ to $A^+$.
If $\pi_A: A^+\to S$ is the projection associated to the generating set $A$,
then $\iota(S,A) = \ker \pi_A$.
Similarly, the word problem of a group $G$ with respect to a generating set $X$ is the kernel of the
projection from the free group on $X$ to $G$.  Kernels of semigroup homomorphisms are relations rather
than sets, which explains the apparent difference between the two definitions.

A \emph{two-tape asynchronous finite state automaton} (AFSA) is a tuple $\A = (Q,q_0,A,F,\delta)$ with
\begin{itemize}
\item $Q$ a finite set of \emph{states} with distinguished start state $q_0\in Q$,
\item $A$ a finite set called the \emph{alphabet},
\item $F\subseteq A$ the set of \emph{final states},
\item the \emph{transition relation} $\Delta\subseteq Q\times (A\cup \{\epsilon\}) \times (A\cup \{\epsilon\})\times Q$.
\end{itemize}

The automaton $\A$ takes as input a pair of words $(u,v)\in A^+\times A^+$.  Starting at $q_0$, the automaton reads
the pair of words asynchronously from left to right, meaning that at any step it can read the `current' symbol from 
either tape, depending on what is allowed by the transition relation $\Delta$: if $\A$ is in state $p$,
then for any $(p,a,b,q)\in \Delta$, the automaton can move to state $q$ if the current symbols on the first and second tapes
are $a$ and $b$ respectively.  If $a$ or $b$ is $\epsilon$, this means that we do not consume input from the corresponding tape
in moving to $q$.  A pair $(u,v)$ is accepted by $\A$ if it is possible for $\A$ to finish in a final state after reading $(u,v)$.
 
A relation $\rho\subseteq A^+\times A^+$ is \emph{rational} if it is recognised by an AFSA (that is, if there is an AFSA $\A$ such that 
the set of all pairs of words accepted by $\A$ is precisely $\rho$).
For word problems of semigroups, the property of being rational is independent of the choice of finite generating set \cite[Corollary~5.4]{NPR}
and hence we can say that a semigroup has \emph{rational word problem} if its word problem with respect to some generating set
(and hence all generating sets) is rational.  As already mentioned, we denote the class of all semigroups with rational word problem by $\rat$.

For any finite semigroup, an AFSA recognising the word problem can be constructed using the semigroup's (right) Cayley graph \cite[Theorem~4.2]{NPR}.
Some obvious examples of infinite semigroups in $\rat$ are the free semigroup $A^+$ and the 
free monoid $A^*$ on a finite set $A$ \cite[Example~4.3]{NPR}.
The word problem of $A^*$ is recognised by an automaton with alphabet $A$, a single initial and final state $q_0$, and transition relation 
$\Delta = \{(q_0,a,a,q_0) \mid a\in A\}$.  
Semigroups in $\rat$ need not be finitely presentable:
Neunh\"offer, Pfeiffer and Ru\v{s}kuc give as an example the semigroup with presentation $\langle a,b \mid (ab^na = aba)_{n\geq 2} \rangle$
\cite[Example~4.4]{ NPR}.

Some examples of semigroups not in $\rat$ are: infinite groups, free commutative semigroups 
and the bicyclic monoid $B = \operatorname{Mon} \langle b,c \mid bc=1 \rangle$ \cite[Theorem~7.4 and Lemmas~4.5 and 4.6]{NPR}.

\section{Monogenic inverse semigroups}\label{mono}

Let $S$ be a semigroup and $u\in S$.  An element $v\in S$ is an \emph{inverse} of $u$ if $uvu = u$ and $vuv = v$.
A semigroup is \emph{regular} if every element has at least one inverse, and \emph{inverse} if every element has 
a \emph{unique} inverse.  Inverse semigroups capture 
the idea of `partial symmetry', much as groups capture the idea of symmetry.
Just as every group is isomorphic to a subgroup of the symmetric group on some set, every inverse semigroup is isomorphic to 
a subgroup of the \emph{full symmetric inverse semigroup} on some set: this is the semigroup consisting of all partial 
bijections on the set (under composition) \cite{Law}.

We will only need to know about the monogenic inverse semigroups in order to prove our result.
An inverse semigroup is \emph{monogenic} if it is generated (as an inverse semigroup), by a single element.
We denote the inverse semigroup generated by an element $u$ by $[u]$.  In $[u]$, the element $u$ has a unique inverse
$v$, and $[u]$ is the semigroup generated by $\{u,v\}$.

The remainder of this section consists mainly of a summary of the results we need on the classification of monogenic inverse semigroups.
We will follow the approach of Preston \cite{Pre}.  The classification is also due independently to Conway,
Duncan and Paterson \cite{CDP}, as Preston notes in his paper.

Preston classified the monogenic inverse semigroups by considering their representations as semigroups of bijections.
We need to introduce his terminology (originally due to Munn) for certain types of bijections.

A \emph{finite link of length $s$} is a mapping of the form $a_i\mapsto a_{i+1}$ for $1\leq i\leq s-1$, 
where $a_1,\ldots,a_s$ are distinct elements.  
A \emph{forward link} is a mapping of the form $a_i\mapsto a_{i+1}$ for $i\in \N$, where $a_1, a_2, \ldots$ is a countably infinite 
sequence of distinct elements.  A \emph{backward link} is the inverse of a forward link.

Any two bijections $\lambda: A\rightarrow B$ and $\mu: C\rightarrow D$ such that $A\cap C = A\cap D = B\cap C = B\cap D = \emptyset$
are called \emph{strongly disjoint}.

\begin{lemma} \label{PreL3} {\rm \cite[Lemma~3]{Pre}}
Let $u$ be the union of the strongly disjoint forward links $\lambda_i$, $i\in I$.  Let $\lambda$ be any (specific)
one of these links.  Then $[u]$ is isomorphic to $[\lambda]$.
\end{lemma}

\begin{lemma} \label{PreT3} {\rm \cite[Theorem~3]{Pre}}
The inverse semigroup generated by a forward link or a backward link is isomorphic to the bicyclic monoid.
\end{lemma} 

Preston established that a monogenic inverse semigroup can be classified into one of the following isomorphism types,
where $r$ is a nonnegative integer and $s$ is either a positive integer or $\infty$:

Type $(r,s)$: isomorphic to $[u]$, where $u$ is the strongly disjoint union of a finite link of length $r$
and a permutation of order $s$.  The semigroup $\langle u\rangle$ is a monogenic semigroup of index $r$ and period $s$.

Type $(r, \Fwd)$: isomorphic to $[u]$, where $u$ is the strongly disjoint union of a finite link of length $r$ and a forward link.

Type $FI$: a free monogenic inverse semigroup.

\begin{theorem} \label{PreT7} {\rm \cite[Theorem~7]{Pre}}
Let $[u]$ be a monogenic inverse semigroup, generated by $u$.  Then $u$ is one of the types $(r,s)$, $(r,\Fwd)$, $FI$.
Moreover, these isomorphism types are distinct.
\end{theorem}

This allows us to conclude the existence of one of a small list of subsemigroups in any monogenic inverse semigroup.

\begin{corollary} \label{submono}
Let $[u]$ be a monogenic inverse semigroup, generated by $u$.  Then one of the following holds.
\begin{enumerate}
\item $u$ is a periodic element;
\item $[u]$ has an infinite cyclic subgroup;
\item $[u]$ has a subsemigroup isomorphic to the bicyclic monoid;
\item $[u]$ is the monogenic free inverse semigroup.
\end{enumerate}
\end{corollary}
\begin{proof}
If $[u]$ is of Type $(r,s)$, then $u^r = u^{r+s}$.  
If $[u]$ is of Type $(r,\infty)$ or $(r, \Fwd)$, then without loss of generality $u$ is 
a disjoint union $\lambda \cup \sigma$, where $\lambda$ is a finite link of length $r$ and $\sigma$ is 
a permutation of infinite order in the case of Type $(r,\infty)$, or a forward link in the case of Type $(r, \Fwd)$. 
Since $\lambda^t$ is the empty transformation for $t\geq r$, we have $u^r = \sigma^r$, and similarly $u^{-r} = \sigma^{-r}$.
So the subsemigroup $[u^r]$ of $[u]$ is generated by either an infinite-order permutation, in which case it is isomorphic to $\Z$, 
or by a finite union of strongly disjoint forward links, in which case it is isomorphic to the bicyclic monoid by 
Lemmas~\ref{PreL3} and~\ref{PreT3}. 
\end{proof}

At the end of \cite[Section~2]{Pre}, Preston defines the following model for the free monogenic inverse semigroup $FI$,
which we shall make use of in the proof of our main theorem.

\[F = \{(-l, n, m) \mid l,n\in \N_0, m\in \Z, 0<n+l, -l\leq m\leq n\}\]
endowed with the multiplication
\[ (-l, n, m) (-l', n', m') = (-l\wedge (m-l'), n\wedge (m+n'), m+m'). \]

\section{Inverse semigroups in $\rat$}\label{proof}

We show that inverse semigroups with rational word problem are finite.  
The following two results mean that it suffices to prove this for monogenic inverse semigroups.

\begin{lemma} \label{fgsub} {\rm \cite[Corollary~5.5]{NPR}}
If $S\in \rat$, then every finitely generated subsemigroup of $S$ is also in $\rat$.
\end{lemma}

\begin{lemma} \label{inforder} {\rm \cite[Theorem~7.1]{NPR}}
Every infinite semigroup with rational word problem has an element of infinite order.
\end{lemma}

\begin{theorem} \label{inverse}
An inverse semigroup has rational word problem if and only if it is finite.
\end{theorem}
\begin{proof}
Let $S$ be an infinite inverse semigroup and suppose that $S$ has rational word problem.  By Lemma~\ref{inforder}, $S$ has an element $u$
of infinite order.  By Lemma~\ref{fgsub}, the monogenic inverse semigroup $[u]$, which is the subsemigroup of $S$ generated by $\{u, u^{-1}\}$,
is in $\rat$.

By Corollary~\ref{submono}, $[u]$ must be the monogenic free inverse semigroup, since $u$ is not a periodic element, and a semigroup with rational
word problem cannot contain an infinite subgroup \cite[Theorem~7.4]{NPR} or a subsemigroup isomorphic to the bicyclic monoid $B$ \cite[Lemma~4.6]{NPR}.  
It remains to show that the monogenic free inverse semigroup $FI$ is not in $\rat$.

Let $S = [u]$ be a monogenic free inverse semigroup and suppose that $\iota(S,\{u,u^{-1}\})$ is rational.  
Then there is an asynchronous finite state automaton $\A$ which accepts $(u^n u^{-n} u^n, u^n)$ for every $n\in \N$.
For $n$ greater than the number of states in $\A$, an accepting path for $(u^n u^{-n} u^n, u^n)$ must go into a loop
while reading $u^{-n}$ on the first tape and some portion of $u^n$ on the second tape.  So $\A$ accepts 
$(u^n u^{-(n+i)} u^n, u^{n+j})$ for some $i\in \N$, $j\in \N_0$.  But by calculating in the model $F$, we get
\begin{align*}
u^n u^{-(n+i)} u^n &= \left(0, n, n\right) \left(-(n+i), 0, -(n+i) \right) \left(0, n, n\right) \\
&= (-i, n, -i) (0, n, n) = (-i, n, n-i), 
\end{align*}
which is not equal to $u^k$ for any $k\in \N$, since $i>0$.  Hence $u^n u^{-(n+i)} u^n\neq u^{n+j}$ in $S$
for any $i\in \N$, $j\in \N_0$, contradicting the fact that $\A$ recognises $\iota_S(\{u, u^{-1}\})$.
Therefore $S$ does not have rational word problem.
\end{proof}

A different generalisation of Anisimov's Theorem to inverse semigroups has been obtained by Kambites \cite{Kam}.
He proved that an inverse semigroup $S$ has regular \emph{idempotent problem} (the set of all words over a generating set for $S$ which represent
idempotents in $S$) if and only if $S$ is finite, answering a question of Gilbert and Noonan Heale \cite{GNH}.

\section{Anisimov's Theorem for regular semigroups}

Since writing our respective notes, the author and Kambites have been informed by Carl Rupert that both versions of 
Anisimov's Theorem are also true for regular semigroups in general.  Semigroups with rational word problem are Kleene 
(see for example \cite[Corollary~8.4.3]{Pfe}),
and Rupert showed \cite{Rup} that regular Kleene semigroups are periodic, from which it follows by Lemma~\ref{inforder}
that regular semigroups with rational word problem are finite.  Rupert's proof of the idempotent problem version of 
Anisimov's Theorem for regular semigroups currently exists only in private communication.

\section*{Acknowledgement}
The author was funded by an EPSRC grant EP/H011978/1.

\end{document}